\documentclass{amsart}
\usepackage{amsmath, amsfonts, amssymb, enumerate, hyperref}
\usepackage{xypic}

\newcommand{\mM}{\mathcal{M}}

\newcommand{\fa}{\mathfrak{a}}

\newcommand{\fq}{\mathfrak{q}}

\newcommand{\bfA}{\mathbf{A}}

\newcommand{\bfC}{\mathbf{C}}

\newcommand{\bfF}{\mathbf{F}}

\newcommand{\bfQ}{\mathbf{Q}}
\newcommand{\bfR}{\mathbf{R}}

\newcommand{\bfZ}{\mathbf{Z}}

\newcommand{\Oo}{\mathcal{O}}

\newcommand{\ov}{\overline}

\newcommand{\be}{\begin{equation}}
\newcommand{\ee}{\end{equation}}
\newcommand{\bes}{\begin{equation*}}
\newcommand{\ees}{\end{equation*}}

\newcommand{\bs}{\begin{split}}
\newcommand{\es}{\end{split}}
\newcommand{\bss}{\begin{split*}}
\newcommand{\ess}{\end{split*}}

\newcommand{\bmat}{\left[ \begin{matrix}}
\newcommand{\emat}{\end{matrix} \right]}
\newcommand{\bsmat}{\left[ \begin{smallmatrix}}
\newcommand{\esmat}{\end{smallmatrix} \right]}

\newcommand{\bml}{\begin{multline}}
\newcommand{\eml}{\end{multline}}
\newcommand{\bmls}{\begin{multline*}}
\newcommand{\emls}{\end{multline*}}

\DeclareMathOperator{\Ext}{Ext}

\DeclareMathOperator{\GL}{GL}

\DeclareMathOperator{\Hom}{Hom}

\DeclareMathOperator{\Nm}{Nm}

\newcommand{\tr}{\textup{tr}\hspace{2pt}}

\usepackage[all]{xy}
\usepackage{amsthm}
\usepackage{amssymb, amsfonts}
\SelectTips{cm}{10}\UseTips
\theoremstyle{plain}
\newtheorem{thm}{Theorem}
\newtheorem{prop}[thm]{Proposition}
\newtheorem{cor}[thm]{Corollary}
\newtheorem{lemma}[thm]{Lemma}
\newtheorem{conj}[thm]{Conjecture}

\theoremstyle{definition}

\newtheorem{rem}[thm]{Remark}

\numberwithin{thm}{section}
\numberwithin{equation}{section}

\begin{document}

\title[Bloch-Kato for Asai $L$-function]{On the Bloch-Kato conjecture for the Asai $L$-function}
\author{Tobias Berger$^1$} 
\address{$^1$School of Mathematics and Statistics, University of Sheffield, Hicks Building, Hounsfield Road, Sheffield S3 7RH, UK.}



\thanks{The author's research was supported by the EPSRC First Grant EP/K01174X/1.}

\begin{abstract}
Following Ribet's seminal 1976 paper there have been many results employing congruences between stable cuspforms and lifted forms to construct non-split extensions of Galois representations. We show how this strategy can be extended to construct elements in the Bloch-Kato Selmer groups of $\pm$-Asai (or tensor induction) representations associated to Bianchi modular forms. We prove, in particular, how the Galois representation associated to a suitable low weight Siegel modular form produces elements in the Selmer group for exactly the Asai representation ($+$ or $-$) that is critical in the sense of Deligne. 

We further outline a strategy using an orthogonal-symplectic theta correspondence to prove the existence of such a Siegel modular form and explain why we expect this to be governed by the divisibility of the near-central critical value of the Asai $L$-function, in accordance with the Bloch-Kato conjecture.
\end{abstract}

\maketitle

\section{Introduction} 
Ribet proved in \cite{Ribet76} that, if a prime $q>2$ divides the numerator of the $k$-th Bernoulli number for $k \geq 4$ an even integer, then $q$ divides the order of the part of the class group of $\bfQ(\mu_q)$ on which ${\rm Gal}(\bfQ(\mu_q)/\bfQ)$ acts by the $(1-k)$-th power of the mod $q$ cyclotomic character $\ov{\epsilon}$. The argument (in fact a slight variance of Ribet's proof outlined in \cite{Khare00}) is as follows: One considers the weight $k$ Eisenstein series for ${\rm SL}_2(\bfZ)$ and proves the existence of a weight $k$ cuspidal eigenform that is congruent to the Eisenstein series. Ribet then shows that one can find a lattice in the irreducible odd  $q$-adic Galois representation $\rho_f$ associated to $f$ such that its mod $q$ reduction gives a non-split extension $\begin{pmatrix} 1&*\\0& \ov{\epsilon}^{k-1}\end{pmatrix}$. Note that this proves (part of) one direction of the Bloch-Kato conjecture for the Riemann zeta function (proven in \cite{BlochKato90} Theorem 6.1 (i)) in so far as it relates the $q$-divisibility of the value $\frac{\zeta(k)}{\pi^k}$ for the even integer $k$ to that of the order of the Selmer group for the critical motive $\bfQ(1-k)$.We note that $\rho_f$ is polarized in the sense that $$\rho_f^{\vee} \cong \rho_f \otimes {\rm det}(\rho_f)^{-1} \cong \rho_f \otimes \epsilon^{1-k},$$ and that the characters occurring in the reduction get swapped under this polarization. 

Ribet's strategy  employing congruences between ``stable'' cuspforms and lifted forms has since been applied to construct many more examples of extensions of one representation by another, i.e. constructing elements in the Selmer groups of tensor product representations. We show in this paper how to go beyond tensor product representations to Asai (or tensor induction) representations.  Using polarized (essentially self-dual) $4$-dimensional Galois representations we provide evidence for the Bloch-Kato conjecture for the Asai $L$-function associated to a Bianchi modular form.

Let $K$ be an imaginary quadratic field with ring of integers $\Oo_K$ and $\pi$ be a cuspidal automorphic representation for ${\rm GL}_2(\bfA_K)$ which is not a base change from $\bfQ$ such that $\pi_{\infty}$ has $L$-parameter 
\begin{align*}
W_{\bfC}=\bfC^\times &\to 
{\rm GL}_2(\bfC): z \mapsto  \begin{pmatrix} (z/|z|)^{k-1}&0\\0&(z/|z|)^{1-k}\end{pmatrix} 
\end{align*} for $k \in \bfZ_{\geq 2}$ and the central character of $\pi$ is trivial. For a prime $q>2$ let $E_{\fq}$ a sufficiently large finite extension of $\bfQ_q$ with ring of integers $\Oo_{\fq}$.

The Asai $L$-function of the title will be defined as a Langlands $L$-function in section \ref{Asaidefinition}. For this introduction we recall  the following explicit description in the special case where $K$ has class number 1 and $k=2$: If $\pi$ has $L$-function $$L(s, \pi)=\sum_{(0) \neq \fa \subset \Oo_K} c(\fa) N(\fa)^{-s}$$ then $$L(s, {\rm As}^+(\pi))=\zeta(2s) \sum_{m=1}^{\infty} \frac{c(m \Oo_K)}{m^s}.$$

We prove in this article the Galois representation half of the argument for the
following part of the Bloch-Kato conjecture \cite{BlochKato90} for the Asai $L$-function:
$$\text{If }\fq\mid L^{\rm int}(1, {\rm As}^{(-1)^k}(\pi)) \text{ then }\fq \mid \# H^1_f(\bfQ,{\rm As}^{(-1)^k}(\rho_{\pi}) (1-k) \otimes E_{\fq}/\Oo_{\fq}),$$ where $\rho_{\pi}:G_K \to {\rm GL}_2(E_{\fq})$ is the Galois representation associated to $\pi$ by Taylor et al. and the (conjecturally motivic) Asai (or tensor induction) representations ${\rm As}^{\pm}(\rho_{\pi})$ are particular extensions of $\rho_{\pi} \otimes \rho_{\pi}^c$ to $G_{\bfQ}$ satisfying ${\rm As}^-(\rho_{\pi}) \cong {\rm As}^+(\rho_{\pi}) \otimes \chi_{K/\bfQ}$, where $\chi_{K/\bfQ}$ denotes the quadratic character  associated to $K/\bfQ$.

Assuming for now the existence of a Galois representation $\tilde R:G_{\bfQ} \to {\rm GL}_4(E_{\fq})$ with $(\tilde R)^{\vee} \cong \tilde R \otimes \Psi$ such that $\tilde R \equiv {\rm ind}_K^{\bfQ}(\rho_{\pi}) \mod{\fq}$ with $\tilde R|_{G_K}$ absolutely irreducible one constructs a suitable extension which proves that the Bloch-Kato Selmer group $H^1_f(K,(\rho_{\pi}^{\vee} \otimes \rho_{\pi}^c) \otimes E_{\fq}/\Oo_{\fq})$ is non-trivial. The key result is Theorem \ref{main} which states that this extension does, in fact,  lift to an element of the Selmer group for the Asai representation of the correct parity, i.e. that it lies in the $(-1)^k$-eigenspace for the action of complex conjugation on $H^1_f(K,{\rm As}^+(\rho_{\pi})(1-k) \otimes E_{\fq}/\Oo_{\fq})$. This eigenspace is isomorphic to  $H^1_f(\bfQ, {\rm As}^{(-1)^k}(\rho_{\pi})(1-k) \otimes E_{\fq}/\Oo_{\fq})$ (see Lemma \ref{Selmerres}), and it is exactly ${\rm As}^{(-1)^k}(\rho_{\pi})$ that is critical at $k$ (in the sense of Deligne) whereas ${\rm As}^{(-1)^{(k+1)}}(\rho_{\pi})$ is not. To prove this we relate in Lemma \ref{explicit} the complex conjugation action to $(-1)^k$ times the  involution arising from the polarization $$(\tilde R|_{G_K})^{c \vee} \cong \tilde R|_{G_K} \otimes \epsilon^{1-k}.$$ For the latter action \cite{BellaicheChenevierbook} Proposition 1.8.10 showed that the eigenvalue is given by the sign of the representation $\tilde R|_{G_K}$, as defined in \cite{BellaicheCheneviersign}. We prove that this sign is $1$ in Lemma \ref{6.4} by deducing that $\tilde R$ is symplectic from its congruence to ${\rm ind}_K^{\bfQ}(\rho_{\pi})$ (which can be equipped with an even or odd symplectic pairing as $\Lambda^2({\rm ind}_K^{\bfQ}(\rho_{\pi}))$ contains two $G_{\bfQ}$-invariant lines under our assumption on the central character of $\pi$). In Lemma \ref{6.4} we  assume that $\Psi$ is congruent to the odd choice of the similitude character for ${\rm ind}_K^{\bfQ}(\rho_{\pi})$, which would be satisfied for a representation coming from a stable Siegel cuspform congruent to the theta lift.

Our result should easily extend to the case of CM fields and $\pi$ satisfying the condition on the central character in \cite{CPMok14}. In fact, the method of proving Theorem \ref{main} extends to general Asai (or tensor induction) representations ${\rm As}^{\pm}(\rho)$ for $\rho:G_{K} \to {\rm GL}_n(E_{\fq})$ for $K/K^+$ a CM field, including extensions of adjoint representations of polarized regular motives as studied in \cite{Harris13}.  This generalization and an application to proving the oddness of residually reducible polarizable Galois representations will be discussed in future work.

For the automorphic half of the argument producing a suitable Galois representation $\tilde R$ as above we suggest replacing the Eisenstein series in Ribet's proof in \cite{Ribet76} by the theta correspondence, i.e. finding a suitable theta series and proving that it is congruent to a ``stable" cuspform whose associated $q$-adic Galois representation is irreducible  and allows one to construct an element in the Selmer group. (A complementary approach using Hermitian Eisenstein series is studied in \cite{Dummigan14}, for which the preceding results of this paper can also be applied, see Remark \ref{rem_dum} for a discussion.) By the Rallis inner product formula (and the expectation that $q$-divisibility of the Petersson norm period ratio  governs congruences to stable forms) one needs to identify a group and an automorphic representation $\sigma$ for which $L^S(1,\sigma, {\rm std})=L^S(1, {\rm As}^{(-1)^k}(\pi))$. By \cite{GI} Lemma 7.1 (see also \cite{Ro01} Lemma 8.1) this is the case for the representation $\sigma=\pi \boxplus \chi_{K/\bfQ}^k$ of $${\rm GSO}(3,1)(\bfA)={\rm GL}_2(\bfA_K) \times \bfA^*/ \{ (z {\rm Id}_2, \Nm_{K/\bfQ} z^{-1}), z \in \bfA_K^*\}.$$  A suitable extension $\hat{\sigma}$ to ${\rm GO}(3,1)$ then has a non-zero theta lift to ${\rm GSp}_4(\bfA)$, as studied by \cite{HST,Ro01, Takeda2009, Takeda2011}, see section \ref{s8} for more details. This theta lift is associated to the Galois representation ${\rm Ind}_K^{\bfQ}(\rho_{\pi})$, %(see e.g. calculation  of unramified Satake parameters in \cite{HST} Lemmas 10 and 11)
in particular it restricts to $\rho_{\pi} \oplus \rho_{\pi}^c$ on $G_K$.

We do not establish the congruence of the theta lift to a stable Siegel cuspform in this paper, but describe in Proposition \ref{nonthetairred} conditions a cuspidal automorphic representation $\Pi$ for ${\rm GSp}_4(\bfA)$ would have to satisfy to give rise to a Galois representation $\tilde R$ as above. Under some additional conditions on $\ov{\rho}_{\pi}$ we prove that one needs to find a $\Pi$ with $\rho_{\Pi} \equiv {\rm ind}_K^{\bfQ} \mod{\fq}$ such that $\Pi$ is not a theta lift of  $\pi \boxplus \chi$ (or any Jacquet-Langlands lifts of $\pi$). That such a $\Pi$ exists is plausible, as one could apply (similar to the proofs for tensor product motives e.g. in \cite{BDSP} and \cite{AgarwalKlosin}) the method using pullback formulas of Siegel Eisenstein series for producing such a congruence to the cohomological Kudla-Millson theta lifts, whose $p$-integrality has been proven in \cite{Berger14}. In particular, \cite{Katsurada08} Lemma 5.1 would allow us to produce a congruence to a non theta-lift, as all the theta lifts appearing in the pullback formula would involve the same Asai $L$-value in the denominator (due to their Petersson norm being expressed by the Rallis inner product formula).
%$\Pi \not \in \Pi(\chi, \pi)$.

To explain why $\Pi$ not being a theta lift implies that the associated $\rho_{\Pi}$ is not isomorphic to ${\rm ind}_K^{\bfQ}(\rho_{\pi})$ (which under the additional assumption that  $\ov{\rho}_{\pi}$ does not deform non-trivially implies that $\rho_{\Pi}|_{G_K}$ is irreducible) we show in Corollary \ref{cor76} that the elements of the global $L$-packet $\Pi(\chi, \pi)$ of the theta lift are indeed characterized by their associated Galois representation being isomorphic to ${\rm ind}_K^{\bfQ}(\rho_{\pi})$. The statement above then follows from work of Roberts and Takeda that shows that all elements of $\Pi(\chi, \pi)$ can be realised via the theta correspondence between a quadratic space of signature $(3,1)$ and ${\rm GSp}_4$ (see Proposition \ref{thetalifts}).

\section*{Acknowledgements}

The author would like to thank Ga\"etan Chenevier, Neil Dummigan, Kris Klosin, Chris Skinner, Jack Thorne and Jacques Tilouine for helpful conversations related to the topics of this article and to Wee Teck Gan, in particular, for comments leading to Proposition \ref{Lparam} and Lemma \ref{nearequiv}.

\section{Asai transfer and tensor induction} \label{Asaidefinition}
We recall from \cite{Krishna12} the definition of the Asai transfer of $\pi$: The Langlands dual of the algebraic group $R_{K/\bfQ} {\rm GL}_2$ is given by ${}^L(R_{K/\bfQ} {\rm GL}_2)=({\rm GL}_2(\bfC) \times {\rm GL}_2(\bfC)) \rtimes G_{\bfQ}$. There are 4-dimensional representations
$$r^{\pm}:{}^L(R_{K/\bfQ} {\rm GL}_2/K)(\bfC)\to {\rm GL}(\bfC^2 \otimes \bfC^2)$$  given by $$r^{\pm}(g,g',\gamma)(x \otimes y)=g(x) \otimes g'(y) \text{ for } \gamma|_K=1$$ and $$r^{\pm}(1,1,\gamma)(x \otimes y)=\pm y \otimes x.\text{ for } \gamma|_K\neq1$$ For each place $v$ we denote the local $L$-group homomorphisms obtained from $r^{\pm}$ by restriction by $r^{\pm}_v$.

Let $\pi$  be a cuspidal automorphic representation for ${\rm GL}_2(\bfA_K)$. Then we may consider $\pi$ as a representation of $R_{K/\bfQ} {\rm GL}_2(\bfA)$ and factorize it as a restricted tensor product $$\pi=\otimes_v \pi_v,$$ where each $\pi_v$ is an irreducible admissible representation of ${\rm GL}_2(K \otimes_{\bfQ} \bfQ_v)$, with corresponding L-parameter $\phi_v: W'_{\bfQ_v} \to ({\rm GL}_2(\bfC) \times {\rm GL}_2(\bfC)) \rtimes G_{\bfQ_v}$. Let ${\rm As}^{\pm}(\pi_v)$ now be the irreducible admissible representation of ${\rm GL}_4(\bfQ_v)$ corresponding to the parameter $r^{\pm}_v \circ \phi_v$ under the local Langlands correspondence and put ${\rm As}^{\pm}(\pi):=\otimes_v {\rm As}^{\pm}(\pi_v)$. By Krishnamurty \cite{Krishna03} and Ramakrishnan \cite{Ramak02} this Asai transfer  ${\rm As}^{\pm}(\pi)$ is known to be an isobaric automorphic representation of ${\rm GL}_4(\bfA)$. 

The corresponding Langlands $L$-function is  defined by $$L(s, \pi, r^{\pm})=\prod_v L(s,\pi_v, r^{\pm}_v)$$ with $L(s,\pi_v, r^{\pm}_v)=L(s, r^{\pm} \circ \phi_v)$.  For unramified $v$ and $\pi_v$ spherical the latter is given by $$\det(I-r^{\pm}(A(\pi_v))p^{-s})^{-1},$$ where $A(\pi_v)$ is the semisimple conjugacy class (Satake parameter) in ${}^L(R_{K/\bfQ} {\rm GL}_2)(\bfC)$.

We will not use ${\rm As}^{\pm}(\pi)$ in the following but rather consider the  corresponding operation for Galois representations, usually called ``tensor (or multiplicative) induction": Consider a representation $\rho:G_K \to {\rm GL}(V)$ for an $n$-dimensional vector space $V$ and write $c$ for the non-trivial element of ${\rm Gal}(K/\bfQ)$. Then we define the following  canonical extensions of $V \otimes V^c$ to $G_{\bfQ}$: 
$${\rm As}^{\pm}(\rho):G_{\bfQ} \to {\rm GL}(V \otimes V^c)$$  given by $${\rm As}^{\pm}(\rho)(g)(x \otimes y)=\rho(g)x \otimes \rho^c(g)(y) \text{ for } g|_K=1$$ and $${\rm As}^{\pm}(\rho)(c)(x \otimes y)=\pm y \otimes x,$$ where $c \in G_{\bfQ}$ is a complex conjugation (which we fix from now on). 
%One checks that different choices of $\tilde c$ give isomorphic representations. 

We first note that ${\rm As}^{-}(\rho)\cong {\rm As}^{+}(\rho) \otimes \chi_{K/\bfQ}$ for $\chi_{K/\bfQ}$ the quadratic character of $G_{\bfQ}$ corresponding to $K/\bfQ$. We quote further properties from \cite{Prasad92} Lemma 7.1:

\begin{lemma}[Prasad] \label{lemPras}
\begin{enumerate}
\item For representations $\rho_1$ and $\rho_2$ of $G_K$, $${\rm As}^{+}(\rho_1 \otimes \rho_2)={\rm As}^{+}(\rho_1) \otimes {\rm As}^{+}(\rho_2).$$
\item $${\rm As}^{\pm}(\rho^{\vee}) \cong {\rm As}^{\pm}(\rho)^{\vee}$$
\item For a one-dimensional representation $\chi$ of $G_K$, ${\rm As}^{+}(\chi)$ is the one-dimensional representation of $G_{\bfQ}$ obtained by composing $\chi$ with the transfer map from $G_{\bfQ}^{\rm ab}$ to $G_K^{\rm ab}$.
\end{enumerate}
\end{lemma}

We also recall from \cite{Krishna12} p.1362-3 how $r^{\pm}$ and ${\rm As}^{\pm}$ are related in the $2$-dimensional case:
Given $\sigma: G_K \to {\rm GL}_2(E)$ (where $E$ is a field) we can define a homomorphism $$\tilde \sigma: G_{\bfQ} \to {}^L(R_{K/\bfQ} {\rm GL}_2)(E)=({\rm GL_2}(E) \times {\rm GL}_2(E)) \rtimes G_K$$ as follows:

$$\tilde \sigma(g)= \begin{cases} (\sigma(g), \sigma(cg c^{-1}), g) & \text{ if } g \in G_K\\ (\sigma(g c^{-1}), \sigma(cg), g) & \text{ if } g \notin G_K.\end{cases}$$

It is now easy to check that $${\rm As}^{\pm}(\sigma)=r^{\pm} \circ \tilde \sigma:G_{\bfQ} \to {\rm GL}_4(E).$$

\section{Relating Selmer groups over $K$ and $\bfQ$}
In this section we record how to relate Bloch-Kato Selmer groups over $K$ and $\bfQ$: 

For $L$ and $E$ number fields and $q$ a prime consider $V$ a continuous finite-dimensional representation of $G_{L}$ over $E_{\fq}$  (with $\fq \mid q$), a finite extension of $\bfQ_q$ with ring of integers $\Oo_{\fq}$. 
 Let $T \subseteq V$ be a
$G_{L}$-stable $\Oo_{\fq}$-lattice and put $W=V/T$ and $W_n=\{x \in W: \fq^nx=0\}$.
Let $\Sigma$ be a finite set of places of $L$. Following Bloch-Kato (see also \cite{DiamondFlachGuo04} Section 2.1) we define for $M=W, W_n$ the following Selmer
groups:

	\be H^1_{\Sigma}(L,M)={\rm ker}(H^1(L,M) \to \prod_{v \notin \Sigma} H^1(L_v,M)/H^1_f(L_v,M)),\ee where 
$H^1_f(L_v, W):={\rm im}(H^1_f(L_v,V) \to H^1(L_v,W))$ and 
$$H^1_f(L_v,V)=\begin{cases} {\rm ker}(H^1(L_v,V) \to H^1(I_v, V)) & \text{ if } v \nmid q, \infty\\ {\rm ker}(H^1(L_v,V) \to H^1(L_v, V \otimes B_{\rm cris})) & \text{ if } v \mid q\\ 0 & \text{ if } v \mid \infty. \end{cases}$$
When $V$ is short crystalline we refer the reader to \cite{DiamondFlachGuo04} Section 2.1 and \cite{BergerKlosin13} Section 5 for the definition of the local subgroups $H^1_f(L_v, W_n)$ for finite modules using Fontaine-Lafaille theory.

\begin{lemma} \label{Selmerres}
Let $V$ be a $G_{\bfQ}$-representation with $T$ and $W$ as above. Let $\Sigma^+$ be a finite set of places of $\bfQ$ and $\Sigma$ the set of places of $K$ above  $\Sigma^+$.  For $q>2$ we have $$H^1_{\Sigma}(K,W) = H^1_{\Sigma}(K,W)^+ \oplus  H^1_{\Sigma}(K,W)^-,$$ where the superscripts indicate the eigenspaces for the action of $c \in G_{\bfQ}$.

Furthermore the restriction map from $G_{\bfQ}$ to $G_K$ induces isomorphisms $$H^1_{\Sigma^+}(\bfQ,W) \overset{\sim}{\to} H^1_{\Sigma}(K,W)^+$$ and $$H^1_{\Sigma^+}(\bfQ,W \otimes \chi_{K/\bfQ}) \overset{\sim}{\to} H^1_{\Sigma}(K,W)^-.$$

\end{lemma}

\begin{proof}
For this we note that ${\rm ind}_K^{\bfQ}(V|_{G_K})=V \oplus V \otimes \chi_{K/\bfQ}$  and recall Shapiro's lemma which for a finite Galois extension of fields $E/F$ tells us that $$H^1(E,M)\cong H^1(F,{\rm ind}_E^F(M))$$ for  any continuous $p$-adic $G_F$-representation $M$ (see \cite{SkinnerUrban14} 3.1.1 and 3.1.2 for a detailed discussion of Shapiro's lemma in the case that $M$ is a discrete module and the compatibility with restrictions at finite places and \cite{SkinnerUrban14} Lemma 3.1 for the analogous statement for the Greenberg Selmer groups in the case of ordinary representations). The statement for  Bloch-Kato Selmer groups follows easily from this and the projection formula $${\rm ind}_{K_v}^{\bfQ_q} (V \otimes B_{\rm cris}|_{G_{K_v}}) \cong {\rm ind}_{K_v}^{\bfQ_q} (V) \otimes B_{\rm cris}.$$ 
\end{proof}

\section{Bloch-Kato Conjecture} \label{s4}
In this section we state the Bloch-Kato conjecture for the Asai $L$-value $L(1,\pi,r^{\pm})$ following the exposition in \cite{Dummigan14}.

Let $\pi$ be a unitary irreducible cuspidal automorphic representation of ${\rm GL}_2(\bfA_K)$ of weight $k$, i.e. with archimedean Langlands parameter $$\begin{pmatrix} (z/|z|)^{k-1} & 0\\0&(|z|/z)^{k-1} \end{pmatrix}.$$ If the central character factors through the norm map from $K$ to $\bfQ$ then for every prime $q$ \cite{HST}, \cite{Taylor94}, \cite{BergerHarcos07} associate to $\pi$ an irreducible Galois representation $\rho_{\pi}:G_K \to {\rm GL}_2(E_{\fq})$ for $E$ a (sufficiently large) finite extension of $\bfQ$ and $\fq \mid q$ (the condition on the central character has been removed by \cite{HLTT} and \cite{Scholze15}).

To ease notation we assume that $\pi$ has trivial central character, which implies in particular, that $\det \rho_{\pi}=\epsilon^{k-1}$ for $\epsilon$ the $q$-adic cyclotomic character (note that we choose the arithmetic Frobenius normalisation). 

Let ${\rm As}^{(-1)^k}(\rho_{\pi}):G_{\bfQ} \to {\rm GL}_2(E_{\fq})$ be the Asai plus/minus representation defined in section \ref{Asaidefinition}.
By \cite{Clozel90}  Conjecture 4.5 (applied to the functorial Asai transfer ${\rm As}^{\pm}(\pi)$ to ${\rm GL}_4(\bfA)$) there should exists a motive $\mM(\pi, r^{(-1)^k})$ of rank 4 with coefficients in $E$ such that its $\fq$-adic realisation is ${\rm As}^{(-1)^k}(\rho_{\pi})^{\vee}$ (dual due to the arithmetic Frobenius normalisation).

\begin{lemma} \label{polarization}
${\rm As}^{\pm}(\rho_{\pi})$ is polarized in the sense that ${\rm As}^{\pm}(\rho_{\pi})^{\vee}(2k-2) \cong {\rm As}^{\pm}(\rho_{\pi})$ or
$${\rm As}^{\pm}(\rho_{\pi}) \otimes \epsilon^{1-k} \cong ({\rm As}^{\pm}(\rho_{\pi}) \otimes \epsilon^{1-k})^{\vee}=\mM(\pi,r^{\pm})(k-1).$$
\end{lemma} 

\begin{proof}
We use the properties of tensor induction stated in Lemma \ref{lemPras}: Put $\rho={\rho_{\pi}}$. Then $ {\rm As}^{\pm}(\rho)^{\vee} \cong {\rm As}^{\pm}(\rho^{\vee})$ and $ {\rm As}^{\pm}(\rho^{\vee})= {\rm As}^{\pm}(\rho \otimes \epsilon^{1-k}|_{G_K})= {\rm As}^{\pm}(\rho) \otimes  {\rm As}^{+}(\epsilon^{1-k}|_{G_K})$. Now note that the transfer of $\epsilon|_{G_K}$ to $G_{\bfQ}$ is $\epsilon^2$.
\end{proof}

We will consider the Tate twist $\mM(\pi,r^{(-1)^k})(k)$ (corresponding to the dual of the Galois representation ${\rm As}^{(-1)^k}(\rho_{\pi}) \otimes \epsilon^{-k}$).
Let $H_{\rm B}(\mM(\pi,r^{\pm})(k))$ and $H_{\rm dR}(\mM(\pi,r^{\pm})(k))$ be the Betti and de Rahm realisations, and let $H_{\rm B}(\mM(\pi,r^{\pm})(k))^{\pm}$ be the eigenspaces for the complex conjugation $F_{\infty}$. Following Deligne we call $\mM(\pi,r^{\pm})(k)$ critical if ${\rm dim}(H_{\rm B}(\mM(\pi,r^{\pm})(k))^+={\rm dim}(H_{\rm dR}(\mM(\pi,r^{\pm})(k))/{\rm Fil}^0)$.
\cite{ghate99} Section 4.2 checked that $\mM(\pi,r^{(-1)^k})(k)$ is critical (our weight $k$ corresponds in his notation to $n=k-2$ and $v_i=v_c=0$).

Assume $q>2k$ and that $\pi$  (here viewed really as representation of $({\rm R}_{K/\bfQ} {\rm GL}_2)(\bfA)$) is not ramified at $q$. Let $\Oo_{\fq}$ be the ring of integers of $E_{\fq}$, and $\Oo_{(\fq)}$ the localisation at $\fq$ of the ring of integers $\Oo_{E}$ of $E$. Write $\varpi$ for a uniformizer of $\Oo_{\fq}$. Choose an $\Oo_{(\fq)}$-lattice $T_B^{\pm}$ in $H_{\rm B}(\mM(\pi,r^{\pm}))$ in such a way that $T_{\fq}^{\pm}:=T_B^{\pm} \otimes \Oo_{\fq}$ is a $G_{\bfQ}$-invariant lattice in the $\fq$-adic realisation. Then choose an  $\Oo_{(\fq)}$-lattice $T_{\rm dR}^{\pm}$ in $H_{\rm dR}(\mM(\pi,r^{\pm}))$ in such a way that $\mathbb{V}(T_{\rm dR}^{\pm} \otimes \Oo_{\fq})=T_{\fq}^{\pm}$ as $G_{\bfQ_p}$-representations, where $\mathbb{V}$ is the version of the Fontaine-Lafaille functor used in \cite{DiamondFlachGuo04}.

Let $\Omega$ be a Deligne period scaled according to the above choice, i.e. the determinant of the isomorphism
$$H_{\rm B}(\mM(\pi,r^{\pm})(k))^+ \otimes \bfC \cong (H_{\rm dR}(\mM(\pi,r^{\pm})(k))/{\rm Fil}^0) \otimes \bfC,$$ calculated with respect to the bases of $(T_B^{\pm})^+$ and $T_{\rm dR}^{\pm}/{\rm Fil^1}$, so well-defined up to  $\Oo_{(\fq)}^*$.

\begin{conj}[\cite{BlochKato90}, \cite{Dummigan14} Conjecture 4.1] \label{BlochKato}
If $\Sigma^+$ is a finite set of primes, containing all $p$ where $\pi_p$ or $K/\bfQ$ is ramified, but not containing $q$ 
then
$${\rm ord}_{\fq} \left( \frac{L^{\Sigma^+}(1,\pi,r^{(-1)^k})}{\Omega}\right)={\rm ord}_{\fq} \left( \frac{\# H^1_{\Sigma^+}(\bfQ,T_{\fq}^{(-1)^k}(k)^*(1) \otimes (E_{\fq}/\Oo_{\fq}))}{\#H^0(\bfQ,T_{\fq}^{(-1)^k}(k)^*(1)\otimes (E_{\fq}/\Oo_{\fq}))}\right),$$ where $T^*_{\fq}={\rm Hom}_{\Oo_{\fq}}(T_{\fq},\Oo_{\fq})$, with the dual action of $G_{\bfQ}$, and $\#$ denotes a Fitting ideal. 
\end{conj}

\begin{rem}
\begin{enumerate}
\item The Langlands $L$-function $L(s,\pi, r^{\pm})$ is expected to have a meromorphic continuation and satisfy a standard functional equation relating $s$ and $1-s$ (and for $L^{\Sigma^+}(s, \pi, r^{\pm})$ this should follow from work of Shahidi). Since ${\rm As}^{\pm}(\rho_{\pi})$ is (conjecturally) pure of weight $2(k-1)$ we have $$L(s, \pi, r^{\pm})=L(s+(k-1),\mM(\pi, r^{\pm})).$$ 

 The value $L(1,\pi,r^{(-1)^k})$ on the LHS in the  conjecture corresponds therefore to the motivic $L$-value $L(k,\mM(\pi, r^{\pm}))=L(0,\mM(\pi, r^{\pm})(k))$

\item By Lemma \ref{polarization} we know that $T_{\fq}(k)^*(1)=T_{\fq}^*(1-k)\cong T_{\fq}(k-1)$, so that the Galois action is given by the self-dual representation ${\rm As}^{(-1)^k}(\rho_{\pi})(1-k)$. We will write the Selmer group in the numerator of the right hand side as  $H^1_{\Sigma^+}(\bfQ, {\rm As}^{(-1)^k}(\rho_{\pi,\fq})(1-k) \otimes (E_{\fq}/\Oo_{\fq}))$ in the following.
\end{enumerate}
\end{rem}

\section{Construction of elements in Selmer group} \label{s5}

This section contains a series of lemmas which together prove the main result Theorem \ref{main} which establishes the Galois part of the proof of one direction of the Bloch-Kato conjecture for the Asai representation.

\begin{thm} \label{main}
Let $q>2$. Assume  that the mod $\fq$-reduction $\ov{\rho}_{\pi}$ is irreducible and does not descend to $G_{\bfQ}$.  
Assume there exists a Galois representation $\tilde R:G_{\bfQ} \to {\rm GL}_4(E_{\fq})$  with the following properties:
\begin{itemize}
\item $\tilde R \equiv {\rm ind}_K^{\bfQ}(\rho_{\pi}) \mod{\fq}$
\item $(\tilde R)^{\vee} \cong \tilde R \otimes \Psi$ for $\Psi:G_{\bfQ} \to E_{\fq}^*$ with $\Psi^{-1} \equiv \chi_{K/\bfQ}^k \epsilon_q^{k-1} \mod{\fq}$ (so in particular $\Psi(c)=-1$),
\item $\tilde R|_{G_K}$ is absolutely irreducible
and is unramified away from a finite set of places $\Sigma$ and short crystalline (for definition see \cite{DiamondFlachGuo04} Section 1.1.2) at $v \mid q$.
\end{itemize}
 
If $\Sigma^+$ is the set of places of $\bfQ$ lying below $\Sigma$  then $$q \mid \# H^1_{\Sigma^+}(\bfQ, {\rm As}^{(-1)^k}(\rho_{\pi,\fq})(1-k) \otimes (E_{\fq}/\Oo_{\fq})).$$ 
 \end{thm}

\begin{proof}
By assumption $(\ov{\tilde R}|_{G_K})^{\rm ss}\equiv \ov{\rho}_{\pi} \oplus \ov{\rho}_{\pi}^c \mod{\fq}$. By Ribet's lemma (see e.g. Theorem 1.1 of \cite{Urban01}) we know there exists a lattice for $\tilde R|_{G_K}$ such that \begin{equation} \label{lattice} \tilde R|_{G_K}\equiv \begin{pmatrix} \ov{\rho}_{\pi} & * \\0 & \ov{\rho}_{\pi}^c \end{pmatrix} \mod{\fq}\end{equation} and this extension is not split.

We claim now that this gives rise to an element in $$H^1_{\Sigma}(K,{\rm As}(\rho)(1-k)\otimes (E_{\fq}/\Oo_{\fq}))^{(-1)^k},$$ which is isomorphic to $H^1_{\Sigma^+}(\bfQ, {\rm As}^{(-1)^k}(\rho_{\pi,\fq})(1-k) \otimes (E_{\fq}/\Oo_{\fq}))$ by Lemma \ref{Selmerres}.

To ease notation put $\rho:=\ov{\rho}_{\pi}$. First note that
\eqref{lattice} gives a non-trivial class in $$H^1(G_K, \Hom_{\bfF}(\rho^c, \rho)).$$  By the assumptions on the ramification and crystallinity of $\tilde R$ we have, in fact, a class in $H^1_{\Sigma}(G_K, \Hom_{\bfF}(\rho^c, \rho))$. Since $$\Hom_{\bfF}(\rho^c, \rho) \cong (\rho^c)^{{\vee}} \otimes \rho \cong (\rho^c \otimes \epsilon^{1-k}) \otimes \rho \cong{\rm As}(\rho)(1-k)|_{G_K}$$
 we obtain an element in $$H^1_{\Sigma}(G_K, {\rm As}(\rho)(1-k)) \cong H^1_{\Sigma}(K,{\rm As}(\rho)(1-k)\otimes (E_{\fq}/\Oo_{\fq})[\fq]),$$ which injects into $H^1_{\Sigma}(K,{\rm As}({\rho}_{\pi,\fq})(1-k) \otimes  (E_{\fq}/\Oo_{\fq}))$ since $H^0(G_{K^+},{\rm As}(\ov{\rho})(1-k))$=0. To see the latter, assume that $\Hom_{G_{K^+}}(\mathbf{1}, {\rm As}(\ov{\rho})(1-k)) \neq 0$. This implies  $\Hom_{G_{K}}(\mathbf{1}, \ov{\rho} \otimes  \ov{\rho}^c \ov{\epsilon}^{1-k}) \neq 0$, contradicting our assumption that $\ov{\rho}$ and $\ov{\rho}^{c}$ are irreducible and non-isomorphic.

%\com{Or should I change to ${\rm Ext}_{\bfF[G]}(\rho^c, \rho)$ for the following?}

\begin{lemma} \label{explicit}
On $H^1(K, {\rm As}^+(\rho)(1-k))$ the complex conjugation action coincides with $(-1)^k$ times the ``polarization involution" arising from the involution on $\bfF[G_K]$ given by $g \mapsto \tau(g):= cg^{-1}c^{-1} \epsilon^{k-1}(g)$ for $g \in G_K$.
\end{lemma}

\begin{proof}
We recall from \cite{BellaicheChenevierbook} Section 1.8 for how $\tau$ induces an involution on ${\rm Ext}^1_{G_K}(\rho^c, \rho)\cong H^1(K, {\rm As}(\rho)(1-k))$: In our case, $\tau$ is an anti-involution of algebras, and the corresponding involution on representations $R: G_K \to {\rm GL}_n(\bfF)$ 
%\com{or $R: \bfF[G_K] \to M_n(B)$ for any $\bfF$-algebra $B$?} 
is given by $R \mapsto R^{\perp}:= (R \circ \tau)^T$. 
We note that we have \begin{equation} \label{eqnpol} P\rho^{\perp}P^{-1} = \rho^c \end{equation} with $P=\begin{pmatrix} 0 & 1\\-1&0 \end{pmatrix}$. 

\cite{BellaicheChenevierbook} (26) on p.51 explains that the induced involution on $H^1_{\Sigma}(K, \Hom_{\bfF}(\rho^c, \rho))$ can be described as follows: 
Associate to a cocycle $\phi \in Z^1(G_K, \Hom_{\bfF}(\rho^c, \rho))$ the  representation $$\rho_{\phi}: G_K \to \GL_4(\bfF): g \mapsto \begin{pmatrix}\rho(g)&b(g)\\0&\rho^c(g) \end{pmatrix}$$ with $b(g)=\phi(g) \rho^c(g)$. Then $\phi^{\perp}$ given by $$\phi^{\perp}(g):=P b^T(cg^{-1}c^{-1}) \epsilon^{k-1} P^{-1} \rho^c(g)^{-1}$$ defines an involution $\perp$ on  $H^1_{\Sigma}(K, \Hom_{\bfF}(\rho^c, \rho))$. (To express $\phi^{\perp}$ directly in terms of $\phi$ apply (\ref{eqnpol}) to write $\phi^{\perp}(g)=\rho(g) P \phi^T(cg^{-1}c^{-1}) P^{-1} \rho^c(g)^{-1}$.) 

We  can rewrite this as follows: \begin{eqnarray*}\phi^{\perp}(g)&=&P b^T(cg^{-1}c^{-1}) \epsilon^{k-1} P^{-1} \rho^c(g)^{-1}\overset{\eqref{eqnpol}}{=}P b^T(cg^{-1}c^{-1}) \rho^c(g)^T P^{-1} \\&=&P \left(\phi(cg^{-1}c^{-1})\rho(g^{-1})\right)^T \rho^c(g)^T P^{-1}=P \left(\rho^c(g) \phi(cg^{-1}c^{-1})\rho(g^{-1})\right)^T  P^{-1}\\&=&P\left(-\phi(cgc^{-1})\right)^T P^{-1},\end{eqnarray*}
using the cocycle relation for the last equality.

We now compare this to the action of $c \in G_{\bfQ}$ on $[\phi] \in H^1_{\Sigma}(K,{\rm As}^+(\rho)(1-k))$ given by $$(c.\phi)(g)={\rm As}(\rho)(1-k)(c) \phi(cgc^{-1}).$$ Since ${\rm As}(\rho)(c)$ acts as transpose on the upper right shoulder, $$\phi^T \in Z^1(G_K, \Hom_{\bfF}( \rho^{\vee}, (\rho^c)^{\vee}),$$ and $$P  \rho^{\vee} P^{-1}=\rho \det(\rho)^{-1},$$  complex conjugation acts by $\phi \mapsto (g \mapsto (-1)^{k-1}P \phi^T (cgc^{-1})P^{-1})$ on  a cocyle representing a class $[\phi] \in  H^1_{\Sigma}(K, {\rm As}^+(\rho)(1-k))$.

\end{proof}

\begin{lemma} \label{6.4}
For the lattice constructed by Theorem 1.1 of \cite{Urban01} the corresponding extension \eqref{lattice} lies in $H^1(\bfQ, {\rm As}(\rho)^{(-1)^k}(1-k))$.
\end{lemma}
 
\begin{proof}
We need to show that the corresponding extension lies in $H^1(K, {\rm As}(\rho)(1-k))^{(-1)^k}$. In fact, Urban's Theorem 1.1  constructs an element $c_G$ of ${\rm Ext}^1_{\bfF[G_K]}(\rho^c, \rho)$ by considering the linear extension $\widehat{\tilde{R}}$ of $\tilde R|_{G_K}$ to $\Oo_{\fq}[G_K] \to M_4(\Oo_{\fq})$.  Put $T:= \tr(\widehat{\tilde{R}})$.  Since $\widehat{\tilde{R}}$ is absolutely irreducible \cite{BellaicheChenevierbook} Proposition 1.6.4 tells us that $\ker \widehat{\tilde{R}}=\ker T$. The morphism $\widehat{\tilde{R}}$ induces a morphism from $\Oo_{\fq}[G_K]/\ker \widehat{\tilde{R}}$ and so, applying \cite{Urban01} Theorem 1.1 again, we see that $c_G$ lies in the subspace  $${\rm Ext}^1_{\bfF[G_K]/\ker T}(\rho^c, \rho)\subset \Ext^1_{\bfF[G_K]}(\rho^c, \rho).$$

By \cite{BellaicheChenevierbook} Proposition 1.8.10(i) we know that $\perp$ acts by multiplication by Bella{\"{\i}}che-Chenevier's ``$U$-sign" ${\rm sign}(\tilde R|_{G_K})={\rm sign}(\tilde R|_{G_K}, \Psi|_{G_K})$ 
on $\Ext^1_{\bfF[G_K]/\ker T}(\rho^c, \rho)$. The $U$-sign of $\tilde R|_{G_K}$ (with respect to $\Psi|_{G_K}$)  is defined as follows (for details see \cite{BellaicheCheneviersign} or \cite{CalegariPoincare}): By assumption we know that $$(\tilde R|_{G_K})^{\vee}= A (\tilde R|_{G_K})^c A^{-1} \Psi|_{G_K}$$ for a character $\Psi:G_{\bfQ} \to E_{\fq}^*$ and an invertible matrix $A$ which can be shown to be either symmetric or antisymmetric. This is used to define the sign by setting  $$A^T= {\rm sign}(\tilde R|_{G_K})  A \text{ for } {\rm sign}(\tilde R|_{G_K}) \in \{ \pm 1 \}.$$ 

To prove that $ {\rm sign}(\tilde R|_{G_K})=1$ we first note that the identity \begin{equation} \label{Lambda} \Lambda^2({\rm ind}_K^{\bfQ}(\rho_{\pi}))={\rm As}^-(\rho_{\pi}) \oplus \epsilon^{k-1} \oplus \chi_{K/\bfQ} \epsilon^{k-1},\end{equation}
shows that the induced Galois representation can be equipped with an odd symplectic pairing.

The following lemma (applied with $\rho_1={\rm ind}_K^{\bfQ}(\rho_{\pi})$) then proves that $\tilde R$ is also symplectic. Together with the oddness of $\Psi$ this implies by  e.g. \cite{CalegariPoincare} Lemma 2.6 that $A$ is symmetric, so concludes the proof of  Lemma \ref{6.4} and the theorem.

\begin{lemma}\label{sign}
Let $q>2$ and $L$ a finite extension of $\bfQ_q$ with ring of integers $\Oo_L$, uniformizer $\pi_L$ and residue field $\bfF_L$. Let $\rho_1, \rho_2:G_{\bfQ} \to {\rm GL}_4(L)$ be two residually absolutely irreducible representations such that $\rho_i^{\vee} \cong \rho_i \otimes \Psi_i$ with $\Psi_i: G_{\bfQ} \to \Oo_L^*$ such that $$\Psi_1 \equiv \Psi_2 \mod{\varpi_L} \text{ and } \ov{\rho}_1 \cong \ov{\rho}_2 \mod{\varpi_L}.$$ Assume that $\ov{\rho}_1$ preserves a non-degenerate $\bfF_L$-bilinear symplectic form. Then so does $\ov{\rho}_2$ and both $\rho_i$ preserve $L$-bilinear symplectic forms.
\end{lemma}
\begin{proof}
We can assume that both representations  are valued in ${\rm GL}_4(\Oo_L)$. By a result of Serre and Carayol we know that there exist $B_i \in {\rm GL}_4(\Oo_L)$ such that $$\rho_i^{\vee} = B_i \rho_i B_i^{-1} \Psi_i.$$ This gives  pairings $\langle \cdot, \cdot \rangle_i: \Oo_L^4 \times \Oo_L^4 \to \Oo_L^*$ defined by $\langle x, y \rangle_i:=x B_i y^T$. We claim that $B_i^T=-B$, i.e. that the signs of $\rho_i$ (again in the sense of \cite{BellaicheCheneviersign}) are $-1$. This follows from \cite{BellaicheCheneviersign} which prove in Section 2.3 for our case of residually absolutely irreducible representations that the signs of $\rho_i$ and $\ov{\rho}_i$ agree (and that the sign of a representation only depends on its isomorphism class). 

For the convenience of the reader we give a direct proof here: Let $M \in {\rm GL}_4(\bfF_L)$ such that $\ov{\rho}_2=M \ov{\rho}_1 M^{-1}$. Then we get that on the one hand $$\ov{\rho}_2^{\vee}=M^{-T} \ov{\rho}_1^{\vee} M^T=M^{-T} (\ov{B}_1 \ov{\rho}_1 \ov{B}_1^{-1}) M^T \ov{\Psi}_1,$$ and on the other hand
$$\ov{\rho}_2^{\vee}=\ov{B}_2 \ov{\rho}_2 \ov{B}_2^{-1} \ov{\Psi}_2=\ov{B}_2 (M \ov{\rho}_1 M^{-1}) \ov{B}_2^{-1} \ov{\Psi}_2.$$
Since $\Psi_1 \equiv \Psi_2$ and $\ov{\rho}_1$ absolutely irreducible Schur's Lemma implies that $$(M^{-T}\ov{B}_1)^{-1} \ov{B}_2M$$ is a scalar matrix, in particular symmetric. This, together with the assumption that $\ov{B}_1^T=-\ov{B}_1$, implies that $\ov{B}_2^T=-\ov{B}_2$,
\end{proof}

\end{proof}

\end{proof}

\begin{rem} \label{rem_dum}
The proof of Theorem \ref{main} only requires the existence of a suitable polarized $G_K$-representation. This can be given by the restriction of a symplectic $G_{\bfQ}$-representation arising from congruences of Siegel modular forms, as discussed in the next sections, but need not to.  As part of general conjectures extending Harder's Eisenstein congruences \cite{Dummigan14} Section 7  discusses such a situation using Eisenstein series  for the quasi-split unitary group $U(2,2)$. 

For $\rho_{\pi}$ arising from a Bianchi modular form of weight $k$ for which a critical (normalised) $L$-value of ${\rm As}^+(\rho_{\pi})$ is divisible by $\fq$  \cite{Dummigan14} conjectures the existence of a cuspidal automorphic representation for $U(2,2)(\bfA)$ whose  associated Galois representation $R:G_K \to {\rm GL}_{4}(E_{\fq})$ should satisfy the following properties: 
\begin{itemize}
\item $R$ is absolutely irreducible and is unramified away from a finite set of places $\Sigma$ and short crystalline at $v \mid q$,
\item $R^{\vee} \cong R^c \otimes \Psi|_{G_K}$ with $\Psi=\epsilon^{1-k} \chi_{K/K^+}^{k}$ (in particular $\Psi(c)=-1$),
\item ${\rm sign}(R)=1$,
\item For an integer $i$ with $0<i\leq \lfloor \frac{k-1}{2}\rfloor$ we have
$$\ov{R}^{\rm ss}=\ov{\rho}_{\pi}(i) \oplus \ov{\rho}_{\pi}^c(i)^{\vee} \otimes \Psi^{-1}|_{G_K}$$ and  $$\ov{\rho}_{\pi}(i) \not \equiv \ov{\rho}_{\pi}^c(i)^{\vee} \Psi^{-1} \mod{\fq}.$$ 
\end{itemize}
(Complementary to this paper the case of $i=0$ corresponding to the near-central $L$-value is excluded.) In this situation one can also apply the arguments in the proof of Theorem \ref{main} to provide evidence for the Bloch-Kato conjecture for all the critical values for ${\rm As}^+(\rho_{\pi})$. \cite{Dummigan14} Section 7 also allows to consider base change forms, in which case the evidence is for symmetric square $L$-values for elliptic modular forms.
\end{rem}

\section{Automorphic Induction and $L$-packets for ${\rm GSp}_4$} \label{s7}
In the remainder of the paper we explain a strategy to procure $\tilde R$ as in theorem \ref{main} as the Galois representation associated to a cuspidal automorphic representation $\Pi$ for ${\rm GSp}_4(\bfA)$.  For this we first need to discuss the automorphic analogue of ${\rm ind}_K^{\bfQ} \rho_{\pi}$:

We briefly recall the definition of $L$-parameters. 
Let $F$ be a number field, $v$ a place of $F$, and $F_v$ the completion of $F$ at $v$. 
Let $G$ be a connected reductive algebraic group over $F$. Then the 
local Langlands correspondence (which is known for $G=\GL_2$, ${\rm GSp}_{4}$ or ${\rm Sp}_{4}$  see \cite{Knapp, BushnellHenniart, GT10, GT2011})
yields a  finite-to-one surjective map (satisfying additional conditions)
$$L: \Pi(G(F_v)) \to \Phi(G(F_v)),$$
where 
\begin{itemize}
\item $\Pi(G(F_v))$ is the set of isomorphism classes of irreducible admissible representations of $G(F_v)$;
\item $\Phi(G(F_v))$ is the set of $L$-parameters for $G(F_v)$, i.e. the set of isomorphism classes 
of admissible homomorphisms $\phi: W_{F_v}' \to {}^LG$, 
where $W_{F_v}'$ is the Weil-Deligne group of $F_v$ and ${}^LG^={}^LG^0\rtimes W_{F_v}$ the $L$-group of $G$ (if $G$ is split over $F$ we take the projection to ${}^LG^0$).
\end{itemize}
For any irreducible admissible representation $\pi_v$ of $G(F_v)$, we call $L(\pi_v)$ the $L$-parameter of $\pi_v$.  

We have an $L$-group homomorphism
$$I^{\bfQ}_K:{}^L(R_{K/\bfQ} {\rm GL}_2/K)(\bfC)\to {\rm GL}(\bfC^2 \oplus \bfC^2)$$  given by $$I^{\bfQ}_K(g,g',\gamma)(x \oplus y)=g(x) \oplus g'(y) \text{ for } \gamma|_K=1$$ and $$I^{\bfQ}_K(1,1,\gamma)(x \oplus y)= y \oplus x.\text{ for } \gamma|_K\neq1.$$
Let $\pi$  be a cuspidal automorphic representation for ${\rm GL}_2(\bfA_K)$ with trivial central character. By the work of Arthur-Clozel we then have its automorphic induction $I^{\bfQ}_K(\pi)$, which is an automorphic representation of ${\rm GL}_4(\bfQ)$ and which is cuspidal if $\pi$ is not a base change. Under our assumptions we know that \begin{equation} \label{Lambdasquare} L(s, \Lambda^2(I^{\bfQ}_K(\pi)) \otimes \chi^{-1})=\zeta(s) L(s, \chi_{K/\bfQ}) L(s,\pi,  r^- \otimes \chi^{-1})\end{equation} has a pole at $s=1$ for $\chi=1$ or $\chi_{K/\bfQ}$, so $I^{\bfQ}_K(\pi)$ descends to a cuspidal representation of ${\rm GSp}_4(\bfQ)$ with central character $\chi=1$ or $\chi_{K/\bfQ}$ (see 
%Kim-Yamaguchi p.18, 
\cite{CPMok14} Proposition 5.1, cf. also discussion towards end of \cite{Ramak04} (correcting sign in \cite{Ramak02})). This uses the non-vanishing of the Asai $L$-function at $s=1$ proved by \cite{Shahidi81}.

These representations lie in  two global $L$-packets for ${\rm GSp}_4$. Following Roberts \cite{Ro01}) we denote them by
$$\Pi(\chi, \pi)=\{\Pi=\otimes_v \Pi_v \in {\rm Irr}_{\rm admiss}({\rm GSp}_4(\bfA)): \Pi_v \in \Pi(\chi_v, \pi_v) \text{ for all  } v\}.$$ 

As indicated, the corresponding $L$-parameters $\varphi(\chi_v, \pi_v):W'_{\bfQ_v} \to {\rm GSp}_4(\bfC)$ composed with ${\rm spin}: {\rm GSp}_4(\bfC) \hookrightarrow {\rm GL}_4(\bfC) $ agree with the $L$-parameters for $I_K^{\bfQ}(\pi)$ (and are distinguished by their similitude being $\chi=\mathbf{1}$ or $\chi_{K/\bfQ}$).

By section 11 of \cite{GT2011} and \cite{CG15} the non-archimedean local packets $\Pi(\chi_v, \pi_v)$ are  also explicitely described by table 4 of the appendix of \cite{Ro01} (via theta correspondence between ${\rm GO}(V)$ for quadratic spaces $V$ over $\bfQ_v$ of discriminant $K \otimes_{\bfQ} \bfQ_v$ and ${\rm GSp}_4$).

We summarize the images of the $L$-parameters of these $L$-packets under various $L$-group homomorphisms in the following proposition:

\begin{prop} \label{Lparam}
For each (finite) place $v \neq 2$ of $\bfQ$ denote by $\varphi_{\pi_v}: W'_{\bfQ_v} \to {}^L(R_{K/\bfQ} {\rm GL}_2)$ the $L$-parameter corresponding to $\pi$ (viewed as a representation of $R_{K/\bfQ} {\rm GL}_2$). We then have:
\begin{enumerate}
\item ${\rm spin} \circ \varphi(\chi_v, \pi_v)=I_K^{\bfQ} \circ \varphi_{\pi_v}$  for ${\rm spin}: {\rm GSp}_4(\bfC) \hookrightarrow {\rm GL}_4(\bfC)$,
\item ${\rm sim} \circ \varphi(\chi_v, \pi_v)=\chi_v$ for ${\rm sim}: {\rm GSp}_4(\bfC) \to {\rm GL}_1(\bfC)$,
\item ${\rm std} \circ \varphi(\chi_v, \pi_v)=\chi_{K/\bfQ,v} \oplus r^+ \circ \varphi_{\pi_v} \otimes \chi_v^{-1} \chi_{K/\bfQ,v}$ for ${\rm std}: {\rm GSp}_4(\bfC) \to {\rm PGSp}_4(\bfC) \cong {\rm SO}_5(\bfC)$.
\end{enumerate}
\end{prop}

\begin{proof}
(1) and (2) are clear from the discussion above.

For unramified places (3) can be shown on the level of Satake parameters,
similar to the proof of \cite{Brown07} Theorem 3.10 using  Lemma 7 of \cite{HST} (but noting the correction given on p.288 of \cite{Ro01}). 

To compute the $L$-parameters directly we can use the fact that for $\phi:W_{\bfQ_v} \to {\rm GSp}_4(\bfC)$ we have $$\Lambda^2 \phi \otimes {\rm sim}(\phi)^{-1}={\rm std} \circ \phi \oplus \bfC$$ (see \cite{GT10} (5.5)).

This allows us to deduce ${\rm std} \circ \varphi(\chi_v, \pi_v)|_{W_{\bfQ_v}}$ from the analogue of (\ref{Lambda}) 
$$\Lambda^2(\varphi(\chi_v, \pi_v)|_{W_{\bfQ_v}}) \otimes \chi_v^{-1}=\bfC \oplus \chi_{K/\bfQ,v} \oplus r^+ \circ \varphi_{\pi_v}|_{W_{\bfQ_v}} \otimes \chi_v^{-1} \chi_{K/\bfQ,v}$$  (which of course also implies (\ref{Lambdasquare})). 

An alternative proof (which also explains better the occurrence of the twist by $\chi_{K/\bfQ}$) can be given by noting that 
the discussion on \cite{Krishna12} pp. 1363/4 shows that $r^+ \circ \varphi_{\pi_v}$ can be equipped with a quadratic form so that it takes values in ${\rm O}_4(\bfC)$: Let $V_K$ be  the space given by the hermitian $2\times 2$ matrices with entries in $K$ with the determinant as its quadratic form.  Then there is a homomorphism $\Phi: {\rm GL}_2(\bfC) \times {\rm GL}_2(\bfC) \to {\rm GSO}(V_K \otimes \bfR)$ defined by $\Phi(A,B):=(X \mapsto AX B^t$). Since $\pi$ was assumed to have trivial central character (so that its $L$-parameters take values in ${\rm SL}_2(\bfC)$) we can modify  $r^+ \circ \varphi_{\pi_v}$ to take values in $${}^L{\rm O}(V_K)= \{ (A,g) \in {\rm O}_4(\bfC) \times W_{\bfQ_v} | {\rm det}(A)=\chi_{K/\bfQ}(g)\}$$ as follows: For $g \in W_{\bfQ_v}$ and $A \in {\rm SL}_2(\bfC)$
 define
$$(r^+ \circ \varphi_{\pi_v})(g,A)=\begin{cases} ((\Phi \times \emptyset)\circ \varphi_{\pi_v}(g, A), g) & \text{ if } \chi_{K/\bfQ}(g)=1\\ ((\Phi \times \emptyset)\circ \varphi_{\pi_v}(g c, A) \Theta, g) & \text{ if } \chi_{K/\bfQ}(g)=-1, \end{cases}$$
where $W_{\bfQ_v}=W_{K_w} \cup  c W_{K_w}$ and $\Theta$ denotes the involution on $V_K \otimes \bfR\cong \bfC^2 \otimes \bfC^2$ corresponding to $x \otimes y \mapsto y \otimes x$.

Appendix C of \cite{GI14} just before Theorem C.5, or \cite{GI15} Theorem 4.4(i)(a), now prove (for $v \neq 2$) that under the theta correpondence between ${\rm O}(V_K)$ and ${\rm Sp}_4$ an $L$-parameter $\phi$ for ${\rm O}(V_K)$ is mapped to the $L$-parameter $\chi_{K/\bfQ} \oplus \phi \otimes \chi_{K/\bfQ}$ valued in ${\rm SO}_5(\bfC)$ (as $\chi_{K/\bfQ}$ is the discriminant character of $V$).

We summarize these relations in the following diagram (which commutes by the description of the packets $\Pi(\chi_v, \pi_v)$ given by \cite{Ro01} Table 4):

$$\xymatrix{  & \ar[ld]  \varphi_{\pi_v} \in \Phi(R_{K/\bfQ} {\rm PGL}_2) \ar[rd]&\\
\varphi(\chi_v, \pi_v) \in \Phi({\rm GSp}_4) \ar[rd]_{{\rm std}} \ar@{^{(}->}[r]_{{\rm spin}\times {\rm sim}}&  \Phi({\rm GL}_4) \times \Phi({\rm GL}_1) \ni (I_K^{\bfQ} \circ \varphi_{\pi_v}, \chi_v)& \Phi({\rm O}(V_K)) \ni r^+ \circ \varphi_{\pi_v}\otimes \chi_v^{-1}\ar[ld] \\
& \chi_{K/\bfQ,v} \oplus r^- \circ \varphi_{\pi_v} \otimes \chi_v^{-1} \in \Phi({\rm Sp}_4)& }
$$ 

\end{proof}

\begin{cor}
For $\Pi \in \Pi(\chi,\pi)$ we have $$L_v(s,\Pi, {\rm spin})=L_v(s, I_K^{\bfQ}(\pi))$$ and $$L_v(s,\Pi, {\rm std})=L_v(s,\chi_{K/\bfQ}) L_v(s,\pi,  r^- \otimes \chi^{-1})$$ for all places $v \neq 2$.  
\end{cor}

\subsection{Characterization of elements of $\Pi(\chi,\pi)$ via Galois representation}
We first  prove that the global packets $\Pi(\chi,\pi)$ are full near equivalence classes:
\begin{lemma} \label{nearequiv}
Assume that $\pi$ is tempered.
Let $\Pi=\bigotimes' \Pi_v$ be an irreducible cuspidal unitary automorphic representation for ${\rm GSp}_4(\bfA_{\bfQ})$ with central character $\chi$. 
If $L_v(s,\Pi, {\rm spin})=L_v(s,I^{\bfQ}_{K} \pi)$ for all but finitely many places $v$ of $\bfQ$ then $\Pi \in \Pi(\chi,\pi)$.
\end{lemma}

\begin{proof}
We adapt the proof of multiplicity one in \cite{Ro01} Theorem 8.5.

By assumption there is a finite set of places $S$ such that for $v \notin S$ the $L$-parameters corresponding to $\Pi$ and $\Pi(\chi, \pi)$ coincide when mapped to ${\rm GL}_4(\bfC)$ via the spin embedding. As the central character of $\Pi$ is assumed to be $\chi$, \cite{GT2011} Lemma 6.1 proves that $\Pi_v \in \Pi(\chi_v, \pi_v)$ for $v \notin S$.

Let $U$ be a space of cuspforms on ${\rm GSp}_4(\bfA)$ realising $\Pi$. Restrict these functions to ${\rm Sp}_4$ and let $U_1$ be a non-zero irreducible subspace. Denote the corresponding representation of ${\rm Sp}_4(\bfA)$ by $\Pi_1$. From Proposition \ref{Lparam}(3) we deduce that the twisted partial standard $L$-function of $\Pi_1$ is 
$$L^S(s,\Pi_1 \otimes \chi_{K/\bfQ}, {\rm std})=\zeta^S(s) L^S(s, \pi, r^+ \otimes \chi^{-1}),$$ so it has a pole at $s=1$ by the non-vanishing of the Asai $L$-function. 

As in \cite{Ro01} p. 307 we can therefore apply \cite{KRS} Theorem 7.1 to conclude that there exists a 4-dimensional quadratic space $X'$ with discriminant $d={\rm disc}(K/\bfQ)$ such that the theta lift $\Theta_X(U_1)$ is a non-zero space of automorphic forms on ${\rm O}(X, \bfA)$. By the proof of \cite{Ro01} Theorem 8.3 this implies that $\Theta_X(U) \neq 0$.
As in \cite{Ro01} p. 307/8 one can further show that $\Theta_X(U)$ is contained in the space of cuspforms for ${\rm GO}(X, \bfA)$. 

By \cite{HST} Proposition 2 (or more generally \cite{Ro01} Section 6)  irreducible cuspidal representations of ${\rm GO}(X, \bfA)$ can be represented by triples $(\pi', \chi', \delta')$. Let $(\pi', \chi', \delta')$ be the data for an irreducible subspace of  $\Theta_X(U)$. Since $\Pi_v \in \Pi(\chi_v, \pi_v)$ for $v \notin S$ we know that  $(\pi'_v, \chi'_v)=(\pi_v, \chi_v)$ for $v \notin S$. By strong multiplicity one for ${\rm GSO}(X, \bfA)$ (see e.g. \cite{HST} Corollary 1) we deduce $\pi'=\pi$ and $\chi'=\chi$.
\end{proof}

\begin{thm} \label{Galreps}
Let $\Pi$ be a cuspidal automorphic representation of ${\rm GSp}_4(\bfA_{\bfQ})$ with central character $\chi$ and $\Pi_{\infty}$ a holomorphic (limit of) discrete series of weight (Blattner parameter) $(k_1, k_2)$ with $k_1 \geq k_2 \geq 2$. For each prime $q$ there exists a continuous semi-simple $q$-adic Galois representation
$$\rho_{\Pi,q}:G_{\bfQ} \to {\rm GL}_4(\ov{\bfQ}_q)$$ which is unramified away from $q$ and the set $S$ of places $v$ where $\Pi_v$ is not spherical such that (for $w=k_1+k_2-3$)
\begin{itemize}
\item $L_v(s- \frac{w}{2},\Pi, {\rm spin})=L_v(s,\rho_{\Pi,q}):={\rm det}(1- \rho_{\Pi,q}({\rm Frob}_v) v^{-s})^{-1}$ for $v \notin S \cup \{q\}$, 
\item $\rho_{\Pi,q}^{\vee}\cong \rho_{\Pi,q} \otimes \chi \epsilon^{-w}$, with $\chi \epsilon^{-w}$ odd,
\item if $\Pi_q$ is an unramified principal series with distinct Satake parameters then $\rho_{\Pi,q}|_{G_{\bfQ_q}}$ is crystalline (and short if $q>k_1+k_2-2$).
\end{itemize}
\end{thm}

\begin{proof}
For $\Pi_{\infty}$ holomorphic discrete series this is due to Taylor, Weissauer and Laumon. In fact, in this case more is known, in particular crystallinity without assumption on Satake parameters, and that the Galois representation preserves a non-degenerate $E_{\fq}$-bilinear symplectic form if $\Pi$ is neither endoscopic or CAP (Weissauer, \cite{BellaicheCheneviersign} Corollary 1.3). The existence of the Galois representation in the limit of discrete series case follows from the discrete series case due to work of Taylor \cite{T91} (or arguments using $p$-adic analytic families, see \cite{Jorza10} and \cite{CPMok14}). Essential self-duality is implied by $\Pi^{\vee} \cong \Pi \otimes \chi$ and the Chebotarev density theorem. The oddness of $\chi \epsilon^{-w}$ follows from the usual parity condition for the non-vanishing of the space of holomorphic Siegel modular forms (see e.g. (2.7) in \cite{Yamauchi}). The statement about crystallinity in the limit of discrete series case has been proven in \cite{Jorza10} Theorem 4.3.4 and (for general CM fields) \cite{CPMok14} Theorem 4.14 and Proposition 4.16.
\end{proof}

\begin{rem}
Note that Lemma \ref{sign} can be used to prove that the image of $\rho_{\Pi,p}$ lies in ${\rm GSp}_4(\ov{\bfQ}_q)$ also for non-regular $\Pi$ such that  $\rho_{\Pi,q}$ is residually absolutely irreducible. This should  follow more generally from recent work of V. Lafforgue \cite{Laf12} Proposition 10.7.
\end{rem}

\begin{cor} \label{cor76}
Let $\Pi=\bigotimes' \Pi_v$ be an irreducible cuspidal unitary automorphic representation for ${\rm GSp}_4(\bfA_{\bfQ})$ with $\Pi_{\infty}$  holomorphic (limit of) discrete series and central character $\chi$. Assume that $\pi$ is tempered.
If the Galois representation associated to $\Pi$ satisfies $\rho_{\Pi} \cong {\rm ind} \rho_{\pi}$ then $\Pi\in \Pi(\chi, \pi)$.
\end{cor}

\begin{prop} \label{nonthetairred}
Let $q>k$ and $\pi$ a tempered cuspidal automorphic representation of ${\rm GL}_2(\bfA_K)$ of weight $k$ and trivial central character such that $\rho_{\pi}$ is crystalline.
%such that $\pi^c \not \equiv \pi$.
Assume %$(q, {\rm cond}(\pi) {\rm disc}(K/\bfQ))=1$, 
$\ov{\rho}_{\pi}:=\rho_{\pi} \mod{\fq}$ is absolutely irreducible and does not descend to $G_{\bfQ}$.

Let $\Pi=\bigotimes' \Pi_v$ be an irreducible cuspidal unitary automorphic representation for ${\rm GSp}_4(\bfA_{\bfQ})$ unramified at $q$ and with $\Pi_{\infty}$  holomorphic (limit of) discrete series of weight $(k_1,k_2)$ with $k_1+k_2-2=k$ and central character $\chi$ with $\rho_{\Pi}$ crystalline. Assume that $\ov{\rho}_{\Pi} \cong {\rm ind}_K^{\bfQ} \ov{\rho}_{\pi} \mod{\fq}$ but that $\Pi \notin \Pi(\chi,\pi)$. 

 Furthermore, assume that $\ov{\rho}_{\pi}$ does not have any non-trivial (short crystalline) characteristic zero deformations of conductor dividing ${\rm cond}(\rho_{\Pi}|_{G_K})$ and determinant $\epsilon^{k_1+k_2-3}$.
Then $\rho_{\Pi}|_{G_K}$ is absolutely irreducible.
\end{prop}

\begin{proof}
By assumption we have that $\ov{\rho}_{\Pi}|_{G_K}=\ov{\rho}_{\pi} \oplus \ov{\rho}^c_{\pi} \mod{\fq}$. 
Assume that $\rho_{\Pi}|_{G_K}$ is reducible (over $E_{\fq}$), i.e. that $(\rho_{\Pi}|_{G_K})^{\rm ss}=\sigma \oplus \sigma'$. By the polarization of $\rho_{\Pi}$ we know that ${\rm det}(\sigma)=\epsilon^{k_1+k_2-3}$ (the other option that $\sigma' \cong \sigma^{\vee} \otimes \epsilon^{-(k_1+k_2-3)}$ cannot occur as $\ov{\rho}_{\pi} \not \equiv \ov{\rho}^c_{\pi}$). Since $\rho_{\pi}$ does not deform non-trivially we get that one of the factors, say $\sigma$, equals $\rho_{\pi}$. By Frobenius reciprocity we have $$0 \neq {\rm Hom}_{G_K}(\rho_{\pi}, \rho_{\Pi}|_{G_K}) \cong {\rm Hom}_{G_{\bfQ}}({\rm ind}_K^{\bfQ}(\rho_{\pi}), \rho_{\Pi}),$$ which by the irreducibility of ${\rm ind}_K^{\bfQ}(\rho_{\pi})$ implies that $\rho_{\Pi} \cong {\rm ind} \rho_{\pi}$. By the previous Corollary this contradicts the assumption that $\Pi \notin \Pi(\chi,\pi)$.
\end{proof}

We note that this Proposition shows that for a cuspidal automorphic representation $\Pi$ as in its statement the Galois representation $\rho_{\Pi}$ associated to it by Theorem \ref{Galreps} would satisfy the assumptions for $\tilde R$ in Theorem \ref{main}. 
It remains to check that $\chi^{-1} \epsilon^{k_1+k_2-3}$ is congruent to $\chi^k_{K/\bfQ} \epsilon^{k-1}$ modulo $\fq$. 
This will be the case as 
%$\rho_{\Pi} \equiv {\rm ind}_K^{\bfQ}(\rho_{\pi}) \mod{\fq}$ implies $k_1+k_2-3 \equiv k-1 \mod{q-1}$ via the determinant and 
we will see in the next section that  $\chi$ has to be $\chi_{K/\bfQ}^k$ to ensure the existence of a suitable holomorphic theta lift.

\begin{rem}
To check the ``no non-trivial deformation" assumption in the proposition one would require ``$R=T$"-theorems generalizing Theorem 5.11 of \cite{CG}. Conjecturally, this would then be governed by $L(1, {\rm Ad}(\pi))$ being a $p$-unit (converse of result due to Urban \cite{U95}).
\end{rem}

\section{Comments on theta correspondence} \label{s8}

As explained in the introduction we propose that a cuspidal automorphic representation $\Pi$ as in Proposition \ref{nonthetairred} can be found using pullback formulas for Siegel Eisenstein series and Katsurada's \cite{Katsurada08} method of proving congruences of lifts with stable forms. For this we note the following result on realising elements of $\Pi(\chi, \pi)$ via the theta correspondence:

\begin{prop}[Roberts, Takeda] \label{thetalifts}
Let $\pi$ be a tempered cuspidal automorphic representation with trivial central character and weight $k \geq 2$.  Assume for $\chi_{\infty}=\mathbf{1}$ and ${\rm sgn}$ that $\Pi(\chi_{\infty}, \pi_{\infty})$ is exhausted  by the theta lifts from ${\rm GO}(3,1)(\bfR)$ with Harish-Chandra parameter $(k-1,0)$ described in \cite{HST} Lemma 12 (3).
Then for $\chi=\mathbf{1}$ or $\chi_{K/\bfQ}$ every element of $\Pi(\chi,\pi)$ occurs in the space of cuspforms on ${\rm GSp}_4(\bfA)$ with central character $\chi$ and can be realised via the theta correspondence between ${\rm GO}(3,1)$ and ${\rm GSp}_4$. In particular,  $\Pi(\chi,\pi)$ contains representations with $\Pi_{\infty}$ a holomorphic limit of discrete series of weight $(k,2)$ if and only if $\chi=\chi_{K/\bfQ}^k$.
\end{prop}
\begin{proof}
Let $\Pi= \bigotimes' \Pi_v \in \Pi(\chi, \pi)$. The proof of Theorem 8.5 of \cite{Ro01} provides a suitable four dimensional quadratic space $X$ and representation $\sigma \in {\rm Irr}^{\rm temp}_{\rm cusp}({\rm GO}(X,\bfA))$ (which is constructed out of $\pi \boxplus \chi$ via Jacquet-Langlands) such that $\theta_2(\sigma_v)^{\vee}=\Pi_v$.
That the global theta lift of $\sigma=\otimes_v \sigma_v$ is non-vanishing follows from \cite{Paul} Corollary 4.17  (archimedean non-vanishing) and \cite{Takeda2009} Theorem 1.2 (local-global nonvanishing).
To show that $\Pi$ is cuspidal automorphic we apply \cite{Takeda2009} Theorem 1.3(2) and \cite{HST} Lemma 5. The statement about the existence of holomorphic representations in the $L$-packet at infinity follows from \cite{HST} Lemma 12 and Corollary 3 (see also \cite{Takeda2011} Proposition 6.5(2)), which shows that it contains two elements, a limit of large and of holomorphic discrete series representation) exactly for $\chi=\chi_{K/\bfQ}^k$.
\end{proof}

We also have the following results on the Petersson norm of the theta lifts:

\begin{lemma}[\cite{Ro01} Lemma 8.1,  \cite{GI} Lemma 7.1] \label{Lstd}
Let $\hat{\sigma}$ be an extension of $\sigma=(\pi, \chi_{K/\bfQ}^k)$ to a cuspidal automorphic representation of ${\rm GO}(3,1)(\bfA)$.
For all finite places $v$ where $\pi$ and $K/\bfQ$  is unramified we have $$L_v(s,\hat{\sigma},{\rm std})=L_v(s, \pi, r^{(-1)^k}).$$
\end{lemma}

\begin{thm}[Rallis inner product formula, \cite{GI} Lemma 7.11, \cite{GQT} Theorem 1.2] 
For $f=\bigotimes f_v \in \pi$ and $\varphi=\bigotimes \varphi_v \in S(V(\bfA))^2$ we have
$$\| \theta(f,\varphi)\|^2=\frac{L^S(1,\hat{\sigma}, {\rm std})}{\zeta^S(2) \zeta^S(4)} \prod_{v \in S} Z_v(f_v, \varphi_v)$$
\end{thm}

Together these results imply that all elements of $\Pi(\chi_{K/\bfQ}^k, \pi)$ have Petersson norm involving $L(1, {\rm As}^{(-1)^k}(\pi))$. \cite{Katsurada08} Lemma 5.1 would (under certain conditions on the  $q$-integrality of the theta lifts and properties of Siegel Eisenstein series) produce congruences of a theta lift with a form with Petersson norm not involving $L(1, {\rm As}^{(-1)^k}(\pi))$. In particular this form 
%could not be a theta lift, and  
would give rise to a $\Pi \notin \Pi(\chi, \pi)$ as required in Proposition \ref{nonthetairred}.

\bibliographystyle{amsalpha}
\bibliography{paramod}

\end{document}